\documentclass[12pt]{amsart}
\usepackage{amsmath,amsthm,latexsym,amscd,amsbsy,amssymb,url}
\usepackage[all]{xypic}
 \relax

\chardef\bslash=`\\ 

\makeatletter
\def\verbatim{\interlinepenalty\@M \@verbatim
  \leftskip\@totalleftmargin\advance\leftskip2pc
  \frenchspacing\@vobeyspaces \@xverbatim}
\makeatother
\newtheorem{thm}{Theorem}[section]

\newtheorem{lem}[thm]{Lemma}

\newtheorem{ex}[thm]{Example}
\newtheorem{que}[thm]{Question}

\begin{document}


\title
{On Monotonic Fixed-Point Free  Bijections on Subgroups of $\mathbb  R$}
\author{Raushan Z.  Buzyakova}
\email{Raushan\_Buzyakova@yahoo.com}

\keywords{ordered group, topological group, homeomorphism, shift,  monotonic function, fixed point, periodic point}
\subjclass{06F15, 54H11, 26A48}


\begin{abstract}{
We show that for any continuous monotonic fixed-point free automorphism  $f$ on a $\sigma$-compact subgroup  $G\subset \mathbb R$ there exists a binary operation $+_f$ such that $\langle G, +_f\rangle$  is a topological group topologically isomorphic to $\langle G, +\rangle$ and $f$ is a shift with respect to $+_f$. We then show that monotonicity cannot be replaced by the property of being periodic-point free. We explore a few routes leading to generalizations and counterexamples.
}
\end{abstract}

\maketitle
\markboth{R. Z. Buzyakova}{On Monotonic Fixed-Point Free Bijections on Subgroups of $\mathbb  R$}
{ }

\section{Introduction}\label{S:intro}

\par\bigskip
In this paper we present a few results in the direction of the following general problem.
\par\bigskip\noindent
{\bf Problem.} {\it Let $G$ be a topological group and let $f$ be  a continuous automorphism on $G$. Is it possible to restructure the algebra of $G$ without changing the topology so that $f$ is a shift, or taking the inverse, or possibly some other function nicely defined in terms of the new binary operation? } 

\par\bigskip\noindent
We show that any fixed-point free monotonic bijection on a $\sigma$-compact subgroup $G\subset \mathbb R$ is a shift with respect to some group structure on $G$ topologically isomorphic to $G$. In particular,
any continuous fixed-point free bijection on the reals $\mathbb R$ is a shift with respect to some group operation on $\mathbb R$ compatible with the Euclidean topology of $\mathbb R$. This result can be used in particular to show that any fixed-point free continuous bijection $f$ on $\mathbb R$ can be colored in three colors. In other words, there exists a cover $\{F_i: i = 1,2,3\}$ such that $f(F_i)\cap F_i=\emptyset$ for each $i=1,2,3$. This colorability fact and an $\epsilon-\delta$ argument for shifts was communicated to the author  by Carlos Nicolas. Our theorem shows that an argument for shifts covers the bijection case as any bijective fixed-point free  map is a shift with respect to some topology-compatible group operation on $\mathbb R$. For a reader interested in coloring, a possible tricolor for the shift $f(x)=x+3$ is $A=\bigcup_{n\in\mathbb Z}[6n,6n+2]$, 
$B=\bigcup_{n\in\mathbb Z}[6n+2,6n+4]$, and $C=\bigcup_{n\in\mathbb Z}[6n+4,6n+6]$. We then discuss failures and successes of certain natural generalization routes. In particular, we show that in our main theorem  monotonicity cannot be replaced by the property of being periodic-point free.

We use standard notations and terminology. For topological basic facts and terminology one can consult \cite{Eng}. Since we do not use any intricate algebraic facts, any abstract algebra textbook is a sufficient reference. We consider only continuous maps. All group shifts   under discussion  are shifts  by a non-neutral element.

\par\bigskip
\section{Study}\label{S:section1}

\par\bigskip
By $\mathbb R$ we denote the set of reals endowed with the Euclidean topology. If a different binary operation or topology is used, it will be specified.
Let us agree that given a bijection $f$ and a positive integer $n$, by $f^n$ we denote the composition of $n$ copies of $f$ and by $f^{-n}$ we denote the composition of $n$ copies of $f^{-1}$. The expression $f^0$ is the identity map. We will use the following folklore facts:
\par\medskip\noindent
{\bf Facts:}
\begin{enumerate}
	\item  If $\langle G, +\rangle$ is a group and $f:G\to G$ is a bijection, then $\langle G, \oplus_f\rangle$ is a group, where $f(x)\oplus_f f(y) = f(x + y)$.
	\item The groups in Fact 1 are isomorphic by virtue of $f$.
	\item If $\langle G, \mathcal T_G, +\rangle$ is a topological group and 
$f:\langle G,\mathcal T_G\rangle\to \langle G,\mathcal T_G\rangle$ is a homeomorphism, then
$\langle G, \mathcal T_G, +\rangle$ and $\langle G, \mathcal T_G, \oplus_f\rangle$ are topologically isomorphic by virtue of $f$.
\end{enumerate}
\par\bigskip\noindent
\begin{thm}\label{thm:shift}
Let $f:\mathbb R\to \mathbb R$ be a continuous  fixed-point free bijection. Then there exists a binary operation $+_f$ on $\mathbb R$ such that $\langle \mathbb R, +_f\rangle$ is a topological group topologically isomorphic to $\mathbb R$ and $f$ is a shift with respect to $+_f$.
\end{thm}
\begin{proof}
Since $f$ is fixed-point free, we conclude that either $f(x)>x$ for all $x$ or $f(x)<x$ for all $x$. We will carry out our argument assuming the former but will make necessary comments for the latter case.
To define our new binary operation, we will first put each $x\in \mathbb R$ into correspondence with $x_f\in \mathbb R$.
\par\medskip\noindent
\underline{\it Definition of $x_f$}. We  define $x_f$ for each $x\in \mathbb R$ in three steps a follows:
\begin{enumerate}
	\item Put $0_f=0$ and $n_f = f^n(0)$ for each $n\in \mathbb Z\setminus \{0\}$.
	\item Let $h:[0,1)\to [0,f(0))$ be an order preserving bijection (hence homeomorphism). For each  $x\in [0,1)$, put $x_f=h(x)$.
\par\smallskip\noindent
{\it Remark. For "$f(x)<x$"-case, $h$ is onto $(f(0),0]$ and order-reversing.}
\par\smallskip\noindent
 Note that $h(0)=0$. Thus, this definition agrees with the first step.
	\item Fix any $x\in \mathbb R$. Then there exist a unique integer $n$ and a unique $x'$ in $[0,1)$ such that $x = x' +n $. Put $x_f = f^n(h(x'))$.
\par\smallskip\noindent
 Note, this definition agrees with steps (1) and (2). Indeed, if $x\in (0,1)$, then $x'=x$ and $n=0$. Then $x_f=f^0(h(x')) = h(x)$, which agrees with step (2).
If $x$ is an integer, then $x'=0$ and $x = 0 +n$. Then $n_f=x_f=f^n(h(0))=f^n(0)$, which agrees with step (1).
\end{enumerate} 

\par\bigskip\noindent
For further reference, we denote by $g$ the correspondence $x\mapsto x_f$.
\par\medskip\noindent
{\it Claim 1}. {\it The correspondence $g$ is an order-preserving bijection on $\mathbb R$ and, hence, a homeomorphism. For "$f(x)<x$"-case, $g$ is order-reversing.}

\par\smallskip\noindent
Let us first show that $g$ is surjective. Fix $y\in \mathbb R$.  Since $f$ is a fixed-point free continuous bijection, $\mathbb R = \bigcup_{n\in \omega}([f^{-(n+1)}(0), f^{-n}(0)]\cup [f^n(0), f^{n+1}(0)])$.  Therefore, there exists $x\in [0, f(0))$ such that $y = f^n(x)$ for some integer $n$. By part (2) of $x_f$-definition, $x=h(z')$ for some $z'\in [0,1)$. Put $z= z' +n$. Then, by part (3), $g(z)=z_f=f^n(h(z'))=f^n(x)=y$. 

Let us show that $g$ is order-preserving. Fix $a,b\in \mathbb R$ such that $a<b$. Let $a=a'+n_a$ and $b=b'+n_b$, where $a',b'\in [0,1)$ and $n_a,n_b\in \mathbb Z$. We assume that $n_a,n_b$ are non-negative. Other cases are treated similarly.
\begin{description}
	\item[\rm Case ($n_a<n_b$)] By part (2) of the definition of $x_f$, we have $h(a'),h(b')\in [0,f(0))$. Since $f$ is a fixed-point free homeomorphism, it is order-preserving. Therefore,  $f^{n_a}(0)\leq f^{n_a}(h(a'))\leq f^{n_a+1}(0)$. Therefore, $a_f\in [f^{n_a}(0),f^{n_a+1}(0)]$. Similarly,  $b_f\in [f^{n_b}(0),f^{n_b+1}(0)]$.
Since $n_b>n_a$, we conclude that $a_f\leq b_f$. Since $f$ and $h$  are one-to-one, we conclude that $a_f<b_f$.
	\item[\rm Case $\neg (n_a<n_b)$] Re-write $b-a>0$ as $(b'-a')+(n_b-n_a)>0$. Since $|b'-a'|<1$, we conclude that $n_b-n_a\geq 0$. By the case's assumption, $n=n_a=n_b$. Hence, $a'<b'$. By part (2) of the definition, $h(a')<h(b')$. Since $f$ is order-preserving, $a_f=f^n(h(a'))<f^n(h(b'))=b_f$.
\end{description}
The claim is proved.
\par\medskip\noindent
Claim 1 and Fact 3 imply that  $\langle \mathbb R, \oplus_g\rangle$ is  a topological group topologically isomorphic to $\mathbb R$ by virtue of $g$. 

\par\medskip\noindent
The next claim completes the proof of the theorem
\par\medskip\noindent
{\it Claim 2.} {\it $f$ is an $\oplus_g$-shift and $f(x_f) = x_f\oplus_g 1_f$ for all $x_f\in \mathbb R$.}
\par\smallskip\noindent
Fix $x_f\in \mathbb R$.
Let $x=x'+n$, where $x'\in [0,1)$ and $n$ is an integer. Then $x_f=f^n(h(x'))$. Therefore, $f(x_f)=f^{n+1}(h(x'))$.
Put $p_f=f^{n+1}(h(x'))$. Then $p=x'+(n+1) = (x'+n)+1$. By the definition of $\oplus_g$, we have $g(p)= g(x'+n)\oplus_g g(1)$, that is, $p_f=x_f\oplus_g1_f$. Since $p_f=f(x_f)$, the claim is proved.

\par\medskip\noindent
To stress the dependence of $\oplus_g$ on $f$ we put $+_f=\oplus_g$, which completes our proof.
\end{proof}

\par\bigskip
Our discussion prompts  a question of whether Theorem \ref{thm:shift} can be generalized to any subgroup $G$ of $\mathbb R$ and any continuous periodic-point free bijection $f$ on $G$. We will show later that a generalization of such a magnitude is not possible. However, certain relaxations on hypotheses can be made. We will next show that the  conclusion of Theorem \ref{thm:shift} holds if we replace $\mathbb R$ by any $\sigma$-compact subgroup of $\mathbb R$ and $f$ by any 
fixed-point free monotonic bijection. We believe that "$\sigma$-compact subgroup" can be replaced by "any subgroup". We will identify the single statement in our argument that requires additional work for a desired generalization. To make argument clearer, let us handle a few cases informally. If $G$ is a  discrete subgroup of $\mathbb R$, then it is order-isomorphic to $\mathbb Z$. Therefore, any  monotonic bijection of $G$ is necessarily a shift. If $G$ has a non-trivial connected component, then $G=\mathbb R$ and Theorem \ref{thm:shift} applies.  To handle the case when  $G$ is  zero-dimensional and dense in $\mathbb R$ let us recall a few classical facts.

It is due to  Sierpienski \cite{Sie} (see also \cite[1.9.6]{vM}) that any countable metric space with no isolated points is homeomorphic to the space of rationals $\mathbb Q$. It is due to Alexandroff and Urysohn \cite{AU} that any $\sigma$-compact zero-dimensional metric space which is nowhere countable is homeomorphic to the product $\mathbb Q\times C$ of the rationals $\mathbb Q$ and the Cantor Set $C$. The immediate applications of these characterizations are the following useful fact:
\begin{description}
	\item[\it Fact] Let $X$ and $Y$ be homeomorphic to $\mathbb Q$ (or both homeomorphic to $\mathbb Q\times C$). Let $a,b\in X$ be distinct and let $c,d\in Y$ be distinct. Then there exists a homeomorphism $h:X\to Y$ such that $h(a)=c$ and $h(b)=d$.
\end{description}
This fact, zero-dimensionality, and homogeneity of $G$ imply the following statement.

\par\bigskip\noindent
\begin{lem}\label{lem:formaintheorem}
Let $G$ be a zero-dimensional dense $\sigma$-compact subgroup of $\mathbb R$ and $U$ a non-empty open subset of $G$. Let $a,b\in U$ be distinct and  $c,d\in G$ satisfy $c<d$. Then there exists a homeomorphism $h:U\to [c,d]\cap G$ such that $h(a)=c$ and $h(b) = d$.
\end{lem}
\begin{proof}
Since $G$ is a group, it is either countable or nowhere countable. Therefore, $G$ is homeomorphic to $\mathbb Q$ or $\mathbb Q\times C$. Since both cases are handled similarly, we assume that the latter is the case. Then any non-empty open subset of $G$ as well as any infinite closed interval of $G$ are $\sigma$-compact and nowhere countable. Therefore, both $U$ and $[c,d]\cap G$ are homeomorphic to $\mathbb Q\times C$. Next apply {\it Fact}.
\end{proof}

\par\bigskip\noindent
The argument of our next result follows that of Theorem \ref{thm:shift}. To avoid unnecessary repetition, we will  reference the already presented argument in a few places. Even though we could have put both theorems under one umbrella, for readability purpose, the author decided to  present them separately.

\par\bigskip\noindent
\begin{thm}\label{thm:monotonic}
Let $G$ be a $\sigma$-compact subgroup of $\mathbb R$ and let $f:G\to G$ be a fixed-point free monotonic bijection. Then there exists a binary operation $+_f$ on $G$ such that $\langle G, +_f\rangle$ is a topological group topologically  isomorphic to $G$ and $f$ is a shift with respect to $+_f$.
\end{thm}
\begin{proof}
 We now may assume that $G$ is a non-discrete zero-dimensional subgroup of $\mathbb R$. Since $G$ is closed under addition and contains a nontrivial sequence converging to $0$, we conclude that  $G$ is dense in $\mathbb  R$.
Since $f$ is a monotonic bijection and $G$ is dense in $\mathbb R$, there exists a continuous bijective extension $\bar f:\mathbb R\to \mathbb R$ of $f$. Let $F$ be the set of all fixed points of $\bar f$. Let $\mathcal J=\{J_n: n = 0, 1,...\}$  consist of all maximal convex sets of $\mathbb R\setminus F$. Note that if $F$ is empty, then $\mathcal J = \{J_0=\mathbb R\}$.  Put $\mathcal I = \{I_n=J_n\cap G:n = 0,1,...\}$.

\par\medskip\noindent
{\it Claim 1.} {\it If $I\in \mathcal I$, then there exists a convex clopen $O\subset I$ in $G$ such that $f^n(O)\cap f^m(O)=\emptyset$ whenever $n\not = m$ and $I = \bigcup_{n\in \mathbb Z} f^n(O)$.}
\par\smallskip\noindent
To prove the claim, fix any point $p$ in $J\setminus G$, where $J\in \mathcal J$ such that $I=J\cap G$. Such a point exists due to zero-dimensionality of $G$. Due to absence of fixed points and  monotonicity of $\bar f$ on $J$, we conclude that $\{f^n(p):n\in \mathbb Z\}$ is unbounded in $J$ on either side. Let $I_p$ be the closed interval in $\mathbb R$ with the endpoints $p$ and $f(p)$. We then have $O=I_p\cap G$  is as desired. The claim is proved.

\par\medskip\noindent
For each $I_n\in \mathcal I$, fix $O_n$ that satisfies the conclusion of the claim. Next for each $x\in G$ we will define $x_f$ as follows.
\par\medskip\noindent
\underline {\it Definition of $x_f$.}  Define $x_f$ in three steps.
\begin{enumerate}
	\item Fix $p^*\in I_0$ and $c\in G$ such that $c>0$. Put $0_f=p^*$ and $(nc)_f = f^n(p^*)$ for each $n\in \mathbb Z$.
	\item Let $I_{p^*}$ be the interval of $G$ with endpoints $p^*$ and $f(p^*)$. Let $A$ and $B$ be convex clopen subsets of $I_{p^*}$   such that $p^*\in A$, $f(p^*)\in B$, and $A\cup B = I_{p^*}$. 
By Lemma \ref{lem:formaintheorem}, there exists a homeomorphism
$h$ of $[0, c]\cap G$ with $A\oplus O_1\oplus...\oplus O_n\oplus ...\oplus B$  such that $h(0)=p^*$ and $h(c)=f(p^*)$. For each $x\in [0,c]\cap G$, put $x_f=h(x)$.

\par\smallskip\noindent
\underline{\it Remark.} Note that our use of Lemma \ref{lem:formaintheorem} is the only part in  which the author could not carry out the argument for an arbitrary zero-dimensional subgroup of $\mathbb R$.

\par\smallskip\noindent
Note that our definition agrees with step 1 for $p^*$ and $f(p^*)$.
	\item For each $x\in G$, there exist a unique $x'\in [0,c)\cap G$ and  a unique $n\in \mathbb Z$ such that $x = x'+nc$. Put $x_f=f^n(h(x'))$. Due to uniqueness of $x'$ and $n$, $x_f$ is well-defined. Note that this definition agrees with our definitions at steps 1 and 2. 
\end{enumerate}
Denote by $g$ the correspondence $x\mapsto x_f$.
\par\medskip\noindent
{\it Claim 2. $g$ is surjective.}
\par\smallskip\noindent
To prove the claim, fix $y\in G$. Then $y\in I_n\in\mathcal I$ for some $n$. By the choice of $O_n$, there exist $m\in \mathbb Z$ and $z\in O_n$ such that $y\in f^m(z)$. If $n=0$, we may assume that $z\not = f(p^*)$. By the definition of $h$, there exists $x'\in [0,c)$ such that $h(x')=z$. Put $x=x'+mc$. Then $g(x)=f^m(h(x'))=f^m(z)=y$. The claim is proved.

\par\medskip\noindent
{\it Claim 3. $g$ is a one-to-one.}
\par\smallskip\noindent
To show that $g$ is one-to-one, fix distinct $a,b\in G$ and let $a=a'+nc$ and $b=b'+mc$, where $a',b'\in[0,c)$ and $n,m\in\mathbb Z$. Then $g(a)=f^n(h(a'))$ and $g(b)=f^m(h(b'))$. If $n=m$, then $g(a)\not  =g(b)$ because both $h$ and $f^n$ are one-to-one. Assume now that $n\not = m$. Let $h(a')\in O_i$ and $h(b')\in O_j$. Assume that $i=j$. Then $g(a)\in f^n(O_i)$ and $g(b)\in f^m(O_i)$. By the choice of $O_i$, we have $f^n(O_i)\cap f^m(O_i)=\emptyset$. If $i\not = j$,
then $g(a)\in I_i$ and $g(b)\in I_j$. By the definition of $\mathcal I$, $I_i\cap I_j=\emptyset$, which proves the claim.

\par\medskip\noindent
{\it Claim 4. $g$ is a homeomorphism.}
\par\smallskip\noindent
Observe that if $x\in [cn,c(n+1)]\cap G$, then $g(x)=f^n(h(x-nc))$. Thus,  $g$ is a homeomorphism on $[cn,c(n+1)]\cap G$.
Since $\{[cn,c(n+1)]\cap G\}_n$ forms a locally finite closed cover of $G$ and $g$ is closed on each element of the cover, we conclude that $g$ is a closed map.  By Claims 2 and 3, $g$ is a homeomorphism on $G$. The claim is proved.

\par\medskip\noindent
 Denote by $+_f$ the operation $\oplus_g$ defined in Facts 1-3. By Fact 3,
$\langle G, +_f\rangle$ is a topological group topologically isomorphic to $G$.  Following the argument of Theorem \ref{thm:shift}, $f(x_f) = x_f +_f c_f$ for all $x_f\in G$.
\end{proof}

\par\bigskip
It is natural to wonder if  monotonicity of $f$ in Theorem \ref{thm:monotonic} can be replaced by a periodic-point free homeomorphism. Since the latter can be of wild nature, a counterexample is  expected. Before we present our construction, we prove the following lemma.  
\par\bigskip\noindent
\begin{lem}\label{lem:example}
Suppose $f:\mathbb Q\to \mathbb Q$ is not an identity map and $p\in \mathbb Q$ satisfy the following property:

\par\smallskip\noindent
(*) $\forall n>0\exists m>0$ such that $f^{m+1}((p-1/n,p+1/n)_\mathbb Q)$ meets $f^{-m}((p-1/n,p+1/n)_\mathbb Q)$.
\par\smallskip\noindent
If $\langle \mathbb Q, \oplus\rangle$ is a topological group topologically isomorphic to $\mathbb Q$, then $f$ is not a $\oplus$-shift.
\end{lem}
\begin{proof}
Let $h:\mathbb Q\to \langle \mathbb Q, \oplus\rangle$ be a topological isomorphism . Let $\prec$ be the order on 
$\langle \mathbb Q, \oplus\rangle$ defined by $a\prec b$ if and only if $h^{-1}(a)< h^{-1}(b)$. Clearly, $\langle \mathbb Q, \oplus, \prec\rangle$  is an ordered topological group. Let $f$ be a $\oplus$-shift. Since $f$ is not the identity map,  there exists $c\in \mathbb Q\setminus \{h(0)\}$  such that $f(x) = x\oplus c$. Without loss of generality, assume that $h(0)\prec c$. Fix $a,b\in \mathbb Q$
such that $a \prec p \prec  b\prec (a\oplus c) \prec (p\oplus c)$. Clearly $f^{m+1}((a,b)_\prec )$ misses $f^{-m}((a,b)_\prec)$ whenever $m>0$. Since the topology on $\langle \mathbb Q, \oplus\rangle$ is Euclidean, there exists $N$ such that 
$(p-1/N, p+1/N)_\mathbb Q\subset (a,b)_\prec$. Therefore,  $f^{m+1}((p-1/N, p+1/N))$ misses 
$f^{-m}((p-1/N, p+1/N))$ for any $m>0$. 
\end{proof}
\par\bigskip\noindent
Note that a non-identity homeomorphism with periodic points cannot be a shift in any group structure isomorphic to $\mathbb Q$. Therefore, the importance of monotonicity in Theorem \ref{thm:monotonic} should  be demonstrated by a non-monotonic homeomorphism with no periodic points.
\par\bigskip\noindent
\begin{ex}\label{ex:nonmonotonic}
There exists a periodic-point free homeomorphism $f:\mathbb Q\to \mathbb Q$ such that $f$ is not a shift in  any topological group structure on $\mathbb Q$ topologically isomorphic to  $\mathbb Q$. 
\end{ex}
\par\noindent
{\it Construction.}  All intervals under consideration are in $\mathbb Q$. Therefore, instead of $(a,b)_\mathbb Q$ we write $(a,b)$. We will construct a homeomorphism $f:\mathbb Q\to\mathbb Q$ that satisfies property (*) of Lemma \ref{lem:example} for $p=0$. Let $\epsilon$ be any irrational number strictly between $0$ and $\frac{1}{2\sqrt{2}}$. 
We put $\epsilon = \frac{1}{\sqrt{7}}$ for better visualization.
\par\noindent

\par\medskip\noindent
\underline{\it Step 1.} Let $g_1$ and $g_{-1}$ be any two order-preserving homeomorphisms with the following ranges and domains:
 $$g_1:\left [0-\frac{1}{2^1\sqrt{2}},0+\frac{1}{2^1\sqrt{2}}\right ]\to [1-\epsilon, 1+\epsilon]$$
 $$g_{-1}:[-1-\epsilon, -1+\epsilon]\to\left  [0-\frac{1}{2^1\sqrt{2}},0+\frac{1}{2^1\sqrt{2}}\right ]$$
Such maps exist since the endpoints in all intervals are irrational, and therefore, not in $G$.
\par\medskip\noindent
\underline{\it Step $n>1$.} Let $g_n$ and $g_{-n}$ be any two order-preserving homeomorphisms with the following ranges and domains:
 $$g_n:g_{n-1}\circ ...\circ g_1 \left(\left [0-\frac{1}{2^n\sqrt{2}},0+\frac{1}{2^n\sqrt{2}}\right ]\right )\to [n-\epsilon, n+\epsilon]$$
 $$
g_{-n}:[-n-\epsilon, -n+\epsilon]\to 
g^{-1}_{-(n-1)}\circ ...\circ g_{-1}^{-1}\left (\left  [0-\frac{1}{2^n\sqrt{2}},0+\frac{1}{2^n\sqrt{2}}\right ]\right )
$$
Note that due to irrationality of $\sqrt{2}$ and $\sqrt{7}$, ranges and domains of these homeomorphisms are clopen intervals in $G$.
\par\medskip\noindent
For better visualization, let us write out the domains of the defined function with $\epsilon =  \frac{1}{\sqrt{7}}$. We have $dom (g_1) = \left [-\frac{1}{2\sqrt{2}}, \frac{1}{2\sqrt{2}}\right ]$. If $n\leq -1$, then 
$dom(g_n) = \left [n- \frac{1}{\sqrt{7}}, n+ \frac{1}{\sqrt{7}}\right ]$. If $n\geq 2$, then 
$dom(g_n)\subset \left [(n-1)- \frac{1}{\sqrt{7}}, (n-1)+ \frac{1}{\sqrt{7}}\right ]$. Due to smallness of $ \frac{1}{\sqrt{7}}$, these sets are mutually disjoint. Let us summarize this observation for further reference.
\par\medskip\noindent
{\it Claim. 
If $n,m\in \mathbb Z\setminus \{0\}$ are distinct, then the domains of $g_n$ and $g_m$ are disjoint.
}

\par\medskip\noindent
Next we select special clopen intervals  as follows.
\par\medskip\noindent
\underline{\it Definition of $A_n,B_n$ for $n>0$:} Fix any non-empty clopen intervals $A_n\subset  (n-\epsilon, n+\epsilon)$ and $B_n\subset  (-n-\epsilon, -n+\epsilon)$ with the following properties:
\begin{description}
	\item[\rm P1] $A_n$ misses the  domain of $g_{n+1}$. Such a set exists because the domain of $g_{n+1}$ is a proper clopen interval of $(n-\epsilon, n+\epsilon)$.
	\item[\rm P2] $B_n$ misses the range of $g_{-(n+1)}$.  Such a set exists because the range of $g_{-(n+1)}$ is a proper clopen interval of $(-n-\epsilon, -n+\epsilon)$.
	\item[\rm P3] $g_n\circ ...\circ g_1\circ g_{-1}\circ ...\circ g_{-n}(B_n)$ misses $A_n$.
\end{description}

\par\medskip\noindent
For each $n$, fix a homeomorphism $h_n$ from $A_n$ onto $B_n$.
 Define $g$ as follows:
$$
g(x) = \left\{
        \begin{array}{ll}
             g_n(x) & x {\rm\  is \ in \ the\ domain\ of}\  g_n\  {\rm\  for}\  n\in \mathbb Z\setminus \{0\} \\
           h_n(x)& x\in A_n
        \end{array}
    \right.
$$
\par\medskip\noindent
To show that $g$ is a function on $D=\left (\bigcup_{n\geq 1} A_n\right )\cup \left (\bigcup \{dom(g_n): n\in \mathbb Z\setminus \{0\}\}\right )$, we need to show that any $x\in D$ belongs to exactly one element of  $\{A_n:n =1,2,...\}\cup  \{dom(g_n): n\in \mathbb Z\setminus \{0\}\}$. By Claim, we may assume that $x\in A_n$. By property P1, $x$ is not in  the domain of $g_{n+1}$. Since $A_n\subset  [n-\epsilon, n+\epsilon]$, $x$ is not in the domain of any $g_m$ with $m\not = n+1$. By Claim and selection of $A_n$'s, we conclude that $x\not\in A_m$ for $n\not = m$.

Let us show that $g$ is a homeomorphism between its domain and range. By property P2, $g$ is one-to-one.  The domain of $g$ is the union of the clopen discrete family 
$\{dom(g_n):n\in \mathbb Z\setminus \{0\}\}\cup \{A_n:n\in \mathbb Z^+\}$. Since $g$ is one-to-one and is a homeomorphism on each member of the family, $g$ is a homeomorphism. By property P3, $g$ has no periodic points. 
 The complements of the domain and range of $g$ are clopen proper subsets of $\mathbb Q$ that are unbounded on both sides.
Therefore, $g$ has a homeomorphic extension $f:\mathbb Q\to \mathbb Q$  that has no periodic points.

It remains to show that $f$ and $0$ satisfy the hypothesis of Lemma \ref{lem:example}.
Fix any $n>0$. We have $f^n\left (\left ( 0-\frac{1}{2^n\sqrt{2}},0+\frac{1}{2^n\sqrt{2}} \right )\right )=(n-\epsilon, n+\epsilon)$.
Since $A_n\subset (n-\epsilon, n+\epsilon)$, we conclude that 
$B_n\subset f^{n+1}\left (\left ( 0-\frac{1}{2^n\sqrt{2}},0+\frac{1}{2^n\sqrt{2}} \right )\right )$. On the other and, $B_n\subset (-n,-\epsilon, -n+\epsilon) = f^{-n}\left (\left ( 0-\frac{1}{2^n\sqrt{2}},0+\frac{1}{2^n\sqrt{2}} \right )\right )$. Therefore, $m=n$ is as desired. $\square$

\par\bigskip
Since zero-dimensional spaces may admit very wildly mannered automorphisms, one may ask if our theorem for $\mathbb  R$ can be extended to $\mathbb R^n$. Alas, an example is in order. 
\par\bigskip\noindent
\begin{ex}\label{ex:R3}
There exists a periodic-point free homeomorphism $f:\mathbb R^3\to \mathbb R^3$ such that $f$ is not a shift in  any group structure of $\mathbb R^3$ topologically isomorphic to $\mathbb R^3$. 
\end{ex}
\par\noindent
{\it Construction.} Let $h:\mathbb R^3\to \mathbb R^3$ be the rotation of the space by $\sqrt 2$ degrees about the $z$-axis in the positive direction. Put $S=\{\langle x,y,z\rangle\in\mathbb R^3: x^2+y^2\leq 1\}$. Define $g:S\to S$ by letting 
$g(x,y,z) = \langle x,y,z+1-x^2-y^2\rangle$. In words, $g$ slides vertically every point by the distance equal to the distance from the point to the wall of the cylinder. Thus, the points on the boundary of the cylinder are not moved. Clearly, both $h$ and $g$ are homeomorphisms. Define $f:\mathbb R^3\to\mathbb R^3$ as follows:

$$
f(x,y,z) = \left\{
        \begin{array}{ll}
             h(x,y,z) & \langle x,y,z\rangle \not\in S \\
          g\circ h(x,y,z) & \langle x,y,z\rangle\in S
        \end{array}
    \right.
$$
\par\bigskip\noindent
Since $g$ is the identity on the boundary of the cylinder, we conclude that $g\circ h$ is equal to $h$ on the boundary of $S$. Therefore, $f$  is a homeomorphism. Points on the $z$-axis are slided up by $1$ unit. Points off the $z$-axis undergo a rotation by $\sqrt 2$ degrees. Thus, $f$ has no periodic points. Since the points on the boundary of the cylinder are only rotated,  the set $\{f^n(1,0,0): n\in \omega\}$ is a an infinite subset of the unit disc in the $xy$-plane centered at the origin. Therefore, the set has a
cluster point. However,$\{\langle x,y,z\rangle +n\langle a,b,c\rangle: n\in \omega\}$ is a closed discrete subset of $\mathbb R^3$ for any $\langle x,y,z\rangle, \langle a,b,c\rangle \in \mathbb R^3$. Therefore, $f$ cannot be a shift in any group structure on $\mathbb R^3$ topologically isomorphic to $\mathbb R^3$.
$\square$ 

\par\bigskip\noindent
We would like to finish this study with a few remarks of categorical nature. Assume that  $f:\mathbb R\to \mathbb R$ is a map that is a shift by a non-neutral element with respect to some group operation $+_f$ on $\mathbb R$ that is compatible with the Euclidean topology. Then, $f$ is a homeomorphism and fixed-point free. Therefore, we have a characterization of all maps on $\mathbb R$ that are shifts by non-neutral elements after some refitting of algebraic structure of $\mathbb R$. It is therefore justifiable to view fixed-point free homeomorphisms on $\mathbb R$ as {\it generalized shifts}. 
Similarly, we can define generalized polynomial (trigonometric, etc.) functions as maps in form 
$f(x) = a_n\star x^n \oplus ...\oplus a_0$, for some addition $\oplus$ and   multiplication $\star$ compatible with the topology of $\mathbb R$. Therefore, it would be interesting to consider the following general problem.
\par\bigskip\noindent
{\bf Problem.} {\it
Characterize generalized polynomials, trigonometric functions, and generalized versions of other standard calculus functions.
}

\par\bigskip\noindent
At last, recall that we were unsuccessful in generalizing Theorem \ref{thm:shift} to a desired extent. Therefore, the question is in order.
\par\bigskip\noindent
\begin{que}
Let $G$ be a topological subgroup of $\mathbb R$ and let $f:G\to G$ be a fixed-point free monotonic homeomorphism. Does there exist a binary operation $\oplus$ on $G$ such that $\langle G,\oplus\rangle$ is a topological group topologically isomorphic to $G$ and $f$ is a $\oplus$-shift?
\end{que}

\par\bigskip\noindent
{\bf Acknowledgment.} The author would like to thank the referee for valuable remarks and corrections.

\par

\end{document}